\documentclass{amsart}
\usepackage{amsmath,amscd,amssymb,verbatim}
\usepackage{xypic}

\newtheorem{theorem}{Theorem}[section]
\newtheorem{lemma}[theorem]{Lemma}
\newtheorem{conjecture}[theorem]{Conjecture}
\newtheorem{corollary}[theorem]{Corollary}

\newtheorem{proposition}[theorem]{Proposition}

\theoremstyle{definition}
\newtheorem{definition}[theorem]{Definition}

\theoremstyle{remark}
\newtheorem{remark}[theorem]{Remark}

\numberwithin{equation}{section}

\begin{document}

\title[VARIETIES OF INTERMEDIATE KODAIRA DIMENSION]{ON THE PLURICANONICAL MAPS OF VARIETIES OF INTERMEDIATE KODAIRA DIMENSION }

\author{Xiaodong Jiang}
\address{Department of Mathematics, University of Utah, Salt Lake City, Utah 84112}
\email{jiang@math.utah.edu}



\date{Feb 21, 2011}



\begin{abstract}
In this paper we will prove a uniformity result for the Iitaka
fibration $f:X\rightarrow Y$, provided that the generic fiber has a
good minimal model and the variation of $f$ is zero or that
$\kappa(X)=\mbox{dim}X-1$.
\end{abstract}

\maketitle

\section{Introduction}
One of the main problems in complex projective algebraic geometry is
to understand the structure of pluricanonical maps. Recently, Hacon
and M$^{\rm{c}}$Kernan \cite{hm06}, Takayama \cite{tak06} and Tsuji
\cite{tsu99} have proved a beautiful result stating that there is a
universal constant $r_n$ such that if $X$ is a smooth projective
variety of general type and dimension $n$, then the pluricanonical
map $$\phi_{rK_X}:X\dashrightarrow \mathbb{P}(H^0(X,
\mathcal{O}_X(rK_X)))$$ is birational for all $r\geq r_n$. In
\cite{hm06}, Hacon and M$^{\rm{c}}$Kernan also proposed a related
conjecture for the Iitaka fibration in the case
$\mbox{dim}X>\kappa(X)\geq 0$.

\begin{conjecture}[\mbox{\cite[Conjecture
1.7]{hm06}}]\label{conjecture} Fix $n\in\mathbb{Z}_{>0}$. There is
positive integer $r_n$ with the following property: Let $X$ be a
smooth $n$-dimensional projective variety of non-negative Kodaira
dimension. Then the rational map $\phi_{rK_X}$ is birationally
equivalent to the Iitaka fibration for all sufficiently divisible
integers  $r\geq r_n$.
\end{conjecture}

The purpose of this paper is to prove Conjecture \ref{conjecture}
under the hypotheses that the Iitaka fibration is isotrivial or that
$\kappa(X)=\mbox{dim}X-1$.

\begin{theorem}\label{main thm}
For any positive integers $n,b,k$, there exists an integer
$m(n,b,k)>0$ such that if $f:X\rightarrow Y$ is the Iitaka fibration
with $X$ and $Y$ smooth projective varieties, $\mbox{\rm{dim}}X=n$,
with generic fiber $F$ of $f$ of Kodaira dimension zero, such that
\begin{enumerate}
\item the variation of $f$ is zero;
\item $F$ has a good minimal model;
\item $b$ is the smallest integer such that $h^0(F,bK_F)\neq 0$, and
$\mbox{\rm{Betti}}_{\rm{dim}(E')}(E')\leq k$, where $E'$ is a smooth
model of the cover $E\rightarrow F$ of the generic fiber $F$
associated to $b$-th root of the unique element of $|bK_F|$;
\end{enumerate}
then the pluricanonical map

$$\phi_{mK_X}:X\dashrightarrow \mathbb{P}H^0(X,\mathcal{O}_X(mK_X))$$
is birationally equivalent to $f$, for any $m\in \mathbb{Z}_{>0}$
such that $m$ is divisible by $m(n,b,k)$.
\end{theorem}

\begin{theorem}\label{n-1 thm} Let $X$ be an $n$-dimensional smooth projective
variety of Kodaira dimension $n-1$ with Iitaka fibration $f:X\to Y$.
Then there exists a positive integer $m_n$ depending only on $n$
such that $\phi_{mK_X}$ is birationally equivalent to $f$ for all
positive integers $m$ divisible by $m_n$.
\end{theorem}

Conjecture \ref{conjecture} has been extensively studied. In
\cite{fm00}, Fujino and Mori prove that if $\kappa(X)=1$, then
(\ref{conjecture}) holds under the hypothesis (3) of Theorem
\ref{main thm}. Viehweg and Zhang \cite{vz07} also obtain this
uniformity result for $\kappa(X)=2$ under the same hypothesis. A
related result of \cite{vz07} for 3-folds has been obtained
independently by Ringler \cite{rin07}. For arbitrary Kodaira
dimension, Pacienza \cite{pac07b} recently has given an affirmative
answer to (\ref{conjecture}) assuming that $Y$ is not uniruled, the
Iitaka fibration $f$ has maximal variation and the hypotheses (2)
and (3) of Theorem
\ref{main thm}.\\

 We now
sketch the proof of Theorem 1.2. The main idea is to follow the
approach of \cite{hm06}, \cite{tak06} and \cite{tsu99}. By the
Canonical Bundle Formula (cf. Section 3), there are two
$\mathbb{Q}$-divisors $M_Y$ (the moduli part) and $B_Y$ (the
boundary part) on $Y$, such that for all $i>0$,
$H^0(X,\mathcal{O}_X(ibNK_X))\cong H^0(Y,\mathcal{O}_Y(\lfloor
ibN(K_Y+M_Y+B_Y)\rfloor))$, where $N$ is a positive integer
depending on the hypothesis (3) of Theorem 1.2 and $M_Y$ is
$\mathbb{Q}$-linearly trivial by the hypotheses (1) and (2) (Theorem
\ref{tfae}). In order to prove Theorem 1.2, it remains to bound a
multiple $m$ of $bN$ for which $\phi_{m(K_Y+M_Y+B_Y)}$ is
birational. We first show that there exists such $m$ of the form
$\alpha(\mbox{vol}(Y,K_Y+M_Y+B_Y))^{-1/n'}+\beta$ that
$\phi_{m(K_Y+M_Y+B_Y)}$ is birational, where $n'=\mbox{dim}Y$ and
$\alpha$, $\beta$ are constants depending only on $n$, $b$ and $k$.
Then using techniques developed in \cite{hmx10}, we show that if
$M_Y$ is $\mathbb{Q}$-linearly trivial, vol$(Y,K_Y+M_Y+B_Y)$ can be
bounded from below. Hence $m$ admits a uniform bound.

The main difficulty is that for a very general point $y\in Y$, we
need to construct an effective $\mathbb{Q}$-divisor $D_y$ which is
$\mathbb{Q}$-linearly equivalent to $\lambda(K_Y+M_Y+B_Y)$, where
$\lambda$ depends on vol$(Y,K_Y+M_Y+B_Y)$, such that $y$ is an
isolated non-klt center of $(Y,D_y)$. There is a well established
way for producing divisors with non-klt centers at $y$. The problem
is that the smallest non-klt center $V$ containing $y$ may be of
positive dimension. In order to produce an isolated non-klt center,
we have to cut down the dimension of the non-klt centers. By
\cite{bchm06}, we can assume $Y$ is the log canonical model, so
$K_{Y}+M_{Y}+B_{Y}$ is ample. Then by Subadjunction (see Section 5)
we prove that vol$(V,(K_Y+M_Y+B_Y)|_V)$ is bounded by a number
related to vol$(Y,K_Y+M_Y+B_Y)$. Using techniques developed in
\cite{mck02} (see Section 4), we can produce a new divisor with a
smaller dimensional non-klt center at $y$. Repeating this procedure
at most $n'-1$ times, we get the desired divisor $D_y$, see Section
6.

In Theorem \ref{n-1 thm}, the generic fiber of the Iitaka fibration
$f$ is an elliptic curve, so the hypotheses (2) and (3) of Theorem
\ref{main thm} automatically hold and $12M_Y$ is linearly equivalent
to a base point free divisor on $Y$. Then using techniques in
\cite{hmx10} and the above argument, we prove that \ref{conjecture}
holds true for any $n$-dimensional variety $X$ of Kodaira dimension
$n-1$.

\subsection*{Acknowledgements} The author would like to thank
his advisor Professor Christopher Hacon for suggesting this problem
and many useful discussions. He would also like to thank Professor
Chenyang Xu for useful suggestions.

\section{Preliminaries}
\subsection{Notation and conventions} We work over the complex
number field $\mathbb{C}$. Let $X$ be a normal variety. We say that
two $\mathbb{Q}$-divisor $D_1$, $D_2$ on $X$ are
$\mathbb{Q}$-linearly equivalent ($D_1\sim_{\mathbb{Q}}D_2$) if
there exists an integer $m>0$ such that $mD_i$ are linearly
equivalent. If $D=\sum d_iD_i$ is a $\mathbb{Q}$-divisor, then the
round down of $D$ is $\lfloor D \rfloor =\sum \lfloor d_i \rfloor
D_i$, where $\lfloor d \rfloor $ denotes the largest integer which
is at most $d$, and the round up of $D$ is $\lceil D \rceil=-\lfloor
-D \rfloor$.

A $\mathbf{log}$ $\mathbf{pair}$ $(X,\Delta)$ is a normal variety
$X$ and an effective $\mathbb{Q}$-Weil divisor $\Delta$ on $X$ such
that $K_X+\Delta$ is $\mathbb{Q}$-Cartier. We say that $(X,\Delta)$
is $\mathbf{log}$ $\mathbf{smooth}$ if $X$ is smooth and $\Delta$ is
a $\mathbb{Q}$-divisor with simple normal crossings support. A
projective morphism $\mu:Y\longrightarrow X$ is a $\mathbf{log}$
$\mathbf{resolution}$ of the pair $(X,\Delta)$ if $Y$ is smooth and
$\mu^{-1}(\Delta)\ \cup$ $\{$exceptional set of $\mu\}$ is a divisor
with simple normal crossings support. We write
$K_Y=\mu^*(K_X+\Delta)+\Gamma$ and $\Gamma = \sum a_i \Gamma_i$
where $\Gamma_i$ are distinct reduced irreducible divisors. We call
$a_i$ the $\mathbf{discrepancy}$ of the pair $(X,\Delta)$ at
$\Gamma_i$. The pair $(X,\Delta)$ is $\mathbf{kawamata}$
$\mathbf{log}$ $ \mathbf{terminal}$, klt for short (resp.
$\mathbf{log}$ $\mathbf{canonical}$, lc for short), if there is a
log resolution $\mu : Y\longrightarrow X$ as above such that the
discrepancies of $\Gamma$ are strictly greater than $-1$, i.e.
$a_i>-1$ for all $i$ (resp. $a_i\geq -1$). A subvariety $V$ of $X$
is called a $\mathbf{non}$-$\mathbf{klt}$ $\mathbf{center}$ of
$(X,\Delta)$ if it is the image of a divisor of discrepancy at most
$-1$. The $\mathbf{non}$-$\mathbf{klt}$ $\mathbf{locus}$
Non-klt$(X,\Delta)$ of the pair $(X,\Delta)$ is the union of the
non-klt centers. A non-klt center $V$ is called a $\mathbf{pure}$
$\mathbf{log}$ $\mathbf{canonical}$ $\mathbf{center}$ if
$(X,\Delta)$ is log canonical at the generic point of $V$.

If $D$ is a Weil divisor on a normal projective variety $X$, then
$\phi_D$ denotes the rational map $X\dashrightarrow
\mathbb{P}H^0(X,\mathcal{O}_X(D))$ induced by global sections of
$\mathcal{O}_X(D)$.

\subsection{Volumes and bounded pairs}
\begin{definition}
Lex $X$ be an irreducible projective variety of dimension $n$ and
$D$ be a $\mathbb{Q}$-divisor. The $\mathbf{volume}$ of $D$ is
$$
\mbox{vol}(X,D)=\limsup_{m\rightarrow \infty}
\frac{n!h^0(X,\mathcal{O}_X(mD))}{m^n}.$$ We say that $D$ is
$\mathbf{big}$ if $\mbox{\rm{vol}}(X,D)>0$.
\end{definition}

We refer the reader to \cite{laz1} for further details.
\begin{lemma}[\mbox{\cite[Lemma 2.2]{hm06}}]\label{vol}
Let $X$ be a projective variety, $D$ a divisor such that $\phi_{D}$
is birational with image $Z$. Then the volume of $D$ is at least the
degree of $Z$ and hence at least $1$.
\end{lemma}

\begin{lemma}[\mbox{\cite[Lemma 2.3.4]{hmx10}}]\label{d-bir}
Let $X$ be a normal projective variety of dimension $n$ and let $D$
be a big $\mathbb{Q}$-Cartier divisor on $X$. If $\phi_{D}$ is
birational, then $\phi_{K_X+(2n+1)(D+M)}$ is birational for any
numerically trivial Cartier divisor $M$.
\end{lemma}

\begin{definition}[\mbox{\cite[Definiton 2.4.2]{hmx10}}]
A set $\mathfrak{D}$ of log pairs is $\mathbf{log}$
$\mathbf{birationally}$ $\mathbf{bounded}$ if there is a log pair
$(Z,B)$ and a projective morphism $Z\longrightarrow T$, where $T$ is
of finite type, such that for every $(X,\Delta)\in \mathfrak{D}$,
there is a closed point $t\in T$ and a birational map
$f:Z_t\dashrightarrow X$ such that the support of $B_t$ contains the
support of the strict transform of $\Delta$ and any $f$-exceptional
divisor.
\end{definition}

\begin{theorem}[\mbox{\cite[Theorem 3.1]{hmx10}}]\label{log-bound}
Fix a positive integer $n$ and two constants $A$ and $\delta >0$.
Then the set of log pairs $(X,\Delta)$ satisfying
\begin{enumerate}
\item $X$ is projective of dimension $n$,
\item $(X,\Delta)$ is log canonical,
\item the coefficients of $\Delta$ are at least $\delta$,
\item there is a positive integer $m$ such that $\mbox{\rm{vol}}(X,m(K_X+\Delta))\leq
A$,
\item $\phi_{K_X+m(K_X+\Delta)}$ is birational,

\end{enumerate}
is log birationally bounded.
\end{theorem}

\begin{theorem}[\mbox{\cite[Theorem 1.7]{hmx10}}]\label{dcc}
Fix a set $I \subset [0,1]$ which satisfies the DCC. Let
$\mathfrak{D}$ be a set of log smooth pairs $(X,\Delta)$, which is
log birationally bounded, such that if $(X,\Delta)\in \mathfrak{D}$,
then the coefficients of $\Delta$ belong to $I$. Then the set
$$\{\mbox{\rm{vol}} (X,K_X+\Delta)|(X,\Delta)\in \mathfrak{D}\},$$
satisfies the DCC.
\end{theorem}

\subsection{Multiplier ideals and singularities of pairs}Let $X$ be
a smooth variety. If $D$ is an effective $\mathbb{Q}$-divisor on
$X$, then the $\mathbf{multiplier}$ $\mathbf{ideal}$
$\mathbf{sheaf}$ associated to $D$ is defined to be
$$\mathcal{J}(X,D)=\mu_*\mathcal{O}_{X'}(K_{X'/X}-\lfloor \mu^*D\rfloor)
$$
where $\mu: X'\rightarrow X$ is a log resolution of $(X,D)$. It is
known that a pair $(X,D)$ is
 klt (resp. non-klt) at a point $x$, if and only if
$$\mathcal{J}(X,D)_x=\mathcal{O}_{X,x}\ \ \ (\mbox{\rm{resp.}} \ \ \ \mathcal{J}(X,D)_x\neq \mathcal{O}_{X,x})$$
and a pair is klt if it is klt at each point $x\in X$. A pair
$(X,D)$ is lc at a point $x$, if and only if
$$\mathcal{J}(X,(1-\varepsilon)D)_x=\mathcal{O}_{X,x}$$
for all rational numbers $0<\varepsilon<1$ and a pair is lc if it is
lc at each point $x\in X$. Note that we have the following relation
for non-klt locus
$$\mbox{Non-klt}(X,D)=\mbox{Supp}(\mathcal{O}_X/\mathcal{J}(X,D))_{\mbox{red}}.$$

The following is a useful way to produce non-klt pairs.

\begin{lemma}[\mbox{\cite[Proposition 9.3.2]{laz2}}]
Assume that $X$ is smooth of dimension $n$, and let $D$ be an
effective $\mathbb{Q}$-divisor on $X$. If $\mbox{\rm{mult}}_xD\geq
n$ at some point $x\in X$, then $\mathcal{J}(X,D)$ is non-trivial at
$x$, i.e. $\mathcal{J}(X,D)\subseteq \mathfrak{m}_x$, where
$\mathfrak{m}_x$ is the maximal ideal of $x$.
\end{lemma}

We now recall Nadel's vanishing theorem.
\begin{theorem}[\mbox{\cite[Theorem 9.4.8]{laz2}}]
Let $X$ be a smooth projective variety. Let $D$ be an effective
$\mathbb{Q}$-divisor on $X$, and $L$ a divisor on $X$ such that
$L-D$ is nef and big. Then, for all $i>0$, we have
$$
H^i(X,\mathcal{O}_{X}(K_X+L)\otimes \mathcal{J}(X,D))=0.$$

\end{theorem}

\subsection{Iitaka fibration}
Here we recall some results regarding Iitaka fibrations.

Let $L$ be a line bundle on an irreducible projective variety $X$.
The semigroup $\mathbf{N}(L)$ of $L$ is
$$\mathbf{N}(L)=\{m\in \mathbb{Z}_{>0}|H^0(X,mL)\neq0\}.$$
Assuming $\mathbf{N}(L)\neq(0)$, all sufficiently large elements of
$\mathbf{N}(L)$ are multiples of a largest single natural number
$e=e(L)\geq 1$, which we call the $\mathbf{exponent}$ of $L$. If
$\kappa(X,L)=\kappa\geq 0$, then $\mbox{dim}(\phi_{mL}(X))=\kappa$
for all sufficiently large $m\in \mathbf{N}(L)$.

\begin{theorem}[\mbox{\rm{Iitaka fibrations}, see \cite[Theorem 2.1.33]{laz1}}] Let $X$ be a normal projective variety, and $L$ a line
bundle on $X$ such that $\kappa(X,L)>0$. Then for all sufficiently
large $k\in \mathbf{N}(L)$, there exists a commutative diagram of
rational maps and morphisms
$$\xymatrix{X  \ar @{-->}[d]_{\phi _k} & X_{\infty} \ar[l]_{u_{\infty}}  \ar [d]^{\phi_{\infty}}\\
Y_k & Y_{\infty} \ar @{-->}[l]^{\nu_k}  }$$ where the horizontal
maps are birational and $u_{\infty}$ is a morphism. One has
$\mbox{\rm{dim}}Y_{\infty}=\kappa(X,L)$. Moreover, if we set
$L_{\infty}=u^*_{\infty}L$, and $F$ is a very general fiber of
$\phi_{\infty}$, we have $\kappa(F, L_{\infty}|_F)=0$.
\end{theorem}

In this paper, we only deal with the case $L=\mathcal{O}_X(K_X)$ and
simply write $\kappa(X)=\kappa(X,\mathcal{O}_X(K_X))$. The following
results are important for our induction in the proof of the main
theorem.

\begin{lemma}\label{iitaka}
Let $X$ and $Y$ be smooth projective varieties and \ $T$ an
algebraic variety. Assume that $f:X\rightarrow Y $ is the Iitaka
fibration of $(X,K_X)$ and $\varphi:Y\rightarrow T$ is a surjective
morphism. For a very general closed point $t\in T$, let $
V=\varphi^{-1}(t)$ and $W=f^{-1}(\varphi^{-1}(t))$, then the
restriction morphism $f_W: W\rightarrow V$ is the Iitaka fibration
of $(W,K_W)$.
\end{lemma}
\begin{proof}By assumption, we have the following diagram
$$
\xymatrix{W \ar @{^{(}->}[r] \ar[d]^{f_W} & X \ar[d]^f\\
V  \ar @{^{(}->}[r] \ar[d] & Y \ar[d]^{\varphi}\\
t \ar @{} [r] |{\in} & T } $$ Since $t$ is very general, we may
assume $V$ and $W$ are smooth and the very general fiber of $f_ W$
is just the very general fiber of $f$. Hence, in order to prove that
$f_W$ is the Iitaka fibration, we only need to show dim$V\leq
\kappa(W)$.

Fix an ample divisor $H$ on $Y$, then there exists a positive
integer $m$ such that $mK_X\geq f^*(H)$. Since $V$ is a smooth
fiber, we have $K_X|_W=K_ W$. It follows that $mK_W\geq
f^*_W(H|_V)$, which implies
$$h^0(W,\mathcal{O}_W(imK_W))\geq h^0(V,\mathcal{O}_V(iH|_V))  \ \ \ \ \ \forall\  i\in \mathbb{Z}_{>0}.$$
Since $H|_V$ is ample on $V$, then $\kappa(W)\geq \mbox{dim}V$.
Therefore, $f_W$ is the Iitaka fibration.
\end{proof}

\begin{theorem}[\mbox{\cite[Theorem 4.4]{lai09}}]\label{lai}
Let $X$ be a $\mathbb{Q}$-factorial normal projective variety with
non-negative Kodaira dimension and at most terminal singularities.
Suppose that the general fiber $F$ of the Iitaka fibration has a
good minimal model, then $X$ has a good minimal model.
\end{theorem}

\section{Canonical bundle formula}
In this section, we collect some of the results regarding the direct
image of the relative dualizing sheaf.

Let $X$ and $Y$ be smooth projective varieties and $f: X\rightarrow
Y$ an algebraic fiber space with generic fiber $F$ of Kodaira
dimension zero. Let $b$ be the smallest integer such that the $b$-th
plurigenus $h^0(F,bK_F)$ of $F$ is non-zero. Then there exists a
$\mathbb{Q}$-divisor $L_{X/Y}$ on $Y$ such that
$$\mathcal{O}_Y(\lfloor iL_{X/Y}\rfloor)\cong (f_*\mathcal{O}_X(ibK_{X/Y}))^{**}$$
and
$$H^0(Y,\mathcal{O}_Y(\lfloor ibK_Y+iL_{X/Y}\rfloor))\cong H^0(X,\mathcal{O}_X(ibK_X))$$
for all $i>0$. We may write the divisor $L_{X/Y}$ as
$$L_{X/Y}=L^{ss}_{X/Y}+\Delta,$$
where $L^{ss}_{X/Y}$ is a $\mathbb{Q}$-Cariter divisor, called the
$\mathbf{semistable}$ $\mathbf{part}$ or the $\mathbf{moduli}$
$\mathbf{part}$, and $\Delta$ is an effective $\mathbb{Q}$-divisor,
called the $\mathbf{boundary}$ $\mathbf{part}$. Moreover, if $f$
satisfies the conditions as in \cite[4.4]{fm00}, then $L^{ss}_{X/Y}$
is nef and $\Delta$ has simple normal crossings support. Therefore,
replacing $Y$ by a smooth birational model, we may always assume
that $L^{ss}_{X/Y}$ is nef and
$\Delta$ is a simple normal crossings divisor.\\

 In applications, it
is important to bound the denominator of $L^{ss}_{X/Y}$.
\begin{theorem}[\mbox{\cite[Theorem 3.1]{fm00}}]
Under the above notations and assumptions, let $E\rightarrow F$ be
the cover associated to the $b$-th root of the unique element of
$|bK_F|$. Let $\overline{E}$ be a nonsingular projective model of
$E$ and let $B_m$ be its $m$-th Betti number. Then there is a
natural number $N=N(B_m)$ depending only on $B_m$ such that
$NL^{ss}_{X/Y}$ is a divisor.
\end{theorem}

Let $\Delta=\sum_{P}s_P P$. We have the following result about the
coefficients $s_P$.
\begin{proposition}[\mbox{\cite[Propostion 2.8]{fm00}}]
Under the notations and the assumptions as above, let $N \in
\mathbb{Z}_{>0}$ be such that $NL^{ss}_{X/Y}$ is a Weil divisor.
Then we have
$$L_{X/Y}=L^{ss}_{X/Y}+\sum_{P}s_PP,$$
where $s_P \in \mathbb{Q}$ for every codimension one point $P$ of
$Y$ is such that
\begin{enumerate}
\item For each $P$, there exists $u_P,v_P \in \mathbb{Z}_{>0}$, such that
$0<v_P\leq bN$ and\\
 $s_P=(bNu_P-v_P)/(Nu_P)$.
\item $s_P=0$ if $f^*(P)$ has only canonical singularities or if $X\rightarrow
Y$ has a semistable resolution in a neighbourhood of $P$.
\end{enumerate}
Moreover, $s_P$ depends only on $f|_{f^{-1}(U)}$ where $U$ is an
open set of $Y$ containing $P$.

\end{proposition}

For convenience, we write $M_Y=L^{ss}_{X/Y}/b$ and $B_Y=\Delta /b$,
then all non-zero coefficients of $B_Y$ are contained in
$$A(b,N):=\{\frac{bNu-v}{bNu}|u,v\in\mathbb{Z}_{>0};0<v\leq bN\}\backslash\{0\}.$$

\begin{lemma}[\mbox{\cite[Lemma 1.2]{vz07}}]\label{vz} Under the notations as above, the following hold true.
\begin{enumerate}
\item The set $A(b,N)$ is a DCC set, and one has
$$\frac{1}{bN}\leq \inf A(b,N).$$
\item $(Y,B_Y)$ is log smooth and has klt singularities.
\item The $\mathbb{Q}$-divisor $K_Y+M_Y+B_Y$ is big.
\item For every $s\in \mathbb{Z}_{>0},$ we have $$H^0(Y,\mathcal{O}_Y(\lfloor
sb(K_Y+M_Y+B_Y)\rfloor))\cong H^0(X,\mathcal{O}_X(sbK_X));$$ further
the map $\phi_{sbK_X}$ is birational to the Iitaka fibration $f$ if
and only if $|sb(K_Y+M_Y+B_Y)|$ gives rise to a birational map.
\item $bNM_Y$ is an integral nef Cartier divisor.
\item If $m\in \mathbb{Z}_{>0}$ is divisible by $bN$, then $\lfloor mB_Y \rfloor\geq (m-1)B_Y
$.
\end{enumerate}
\end{lemma}

\begin{lemma}\label{canonical}
Under the same notations and assumptions as in Lemma \ref{vz},
$(Y,M_Y+B_Y)$ has a log terminal model and a log canonical model.
\end{lemma}
\begin{proof}
Since $K_Y+M_Y+B_Y$ is big, we may write
$K_Y+M_Y+B_Y\sim_{\mathbb{Q}}A+E$, where $A$ is an ample
$\mathbb{Q}$-divisor and $E$ is an effective $\mathbb{Q}$-divisor.
By (2) of Lemma \ref{vz}, $(Y,B_Y)$ is klt, so $(Y,B_Y+\epsilon E)$
is also klt for $0<\epsilon\ll 1$. By (5) of Lemma \ref{vz}, $M_Y$
is nef, so $M_Y+\epsilon A$ is ample. Thus there exist a
sufficiently ample divisor $A'$ and a rational number
$0<\epsilon'\ll 1$ such that $M_Y+\epsilon
A\sim_{\mathbb{Q}}\epsilon'A'$ and $(Y,B_Y+\epsilon E+\epsilon'A')$
is also klt. It follows that
\begin{eqnarray*}
(1+\epsilon)(K_Y+M_Y+B_Y)&\sim_{\mathbb{Q}}&K_Y+M_Y+B_Y+\epsilon
A+\epsilon E\\
&\sim_{\mathbb{Q}}&K_Y+B_Y+\epsilon E+\epsilon'A'.
\end{eqnarray*}
By \cite{bchm06}, $(Y,B_Y+\epsilon E+\epsilon'A')$ has a log
terminal model $Y^m$ and a log canonical model $Y^c$. It is easy to
see that $Y^m$ (resp. $Y^c$) is also a log terminal model (resp. log
canonical model) of $(Y,M_Y+B_Y)$.
\end{proof}

\begin{lemma}\label{restriction}
Under the notations and assumptions as in Lemma \ref{iitaka}, the
boundary part $B_V$ of \,$f_W$ is the restriction of $B_Y$ to $V$
and the moduli part $M_V$ of \,$f_W$ is $\mathbb{Q}$-linearly
equivalent to the restriction of $M_Y$.
\end{lemma}
\begin{proof} Since $(Y,B_Y)$ is log smooth and $V$ is a very general fiber of $\varphi:Y\to
T$, we may assume that $B_Y|_V$ has simple normal crossings support.
Let $B_Y=\sum_P r_PP$ and $B_V=\sum_Q r'_QQ$. Recall that $1-r_P$ is
the log canonical threshold of $f^*P$ with respect to $(X,-D_X/b)$
over the generic point of $P$ and $1-r'_Q$ is the log canonical
threshold of $f^*_WQ$ with respect to $(W,-D_W/b)$ over the generic
point of $Q$, where $D_X=bK_X-f^*(bK_Y+L_{X/Y})$ and
$D_W=bK_W-f^*_W(bK_V+L_{W/V})$ (see \cite[Defintion 3.4]{fuj03}).
Since $W$ is a very general fiber, we have $D_X|_W=D_W$. Hence
$r'_Q=0$ when $Q$ is not contained in the support of $B_Y|_V$ and
$r'_Q=r_P$ when $Q$ is the restriction of some component $P$ of
$B_Y$. Therefore $B_V=B_Y|_V$. On the other hand, we have
$K_V+M_V+B_V\sim_{\mathbb{Q}}(K_Y+M_Y+B_Y)|_V$. Hence
$M_V\sim_{\mathbb{Q}}M_Y|_V$.
\end{proof}

\subsection*{Variation}Let $f:X\to Y$ be an algebraic fiber space. Let $K\supset
\mathbb{C}$ be an algebraically closed field contained in
$\overline{\mathbb{C}(Y)}$ such that there is a finitely generated
extension $L$ of $K$ such that
$Q(L\otimes_{K}\overline{\mathbb{C}(Y)})\cong
Q(\mathbb{C}(X)\otimes_{\mathbb{C}(Y)}\overline{\mathbb{C}(Y)})$
over $\overline{\mathbb{C}(Y)}$, where $Q$ denotes the fraction
field. The minimum of tr.deg$_{\mathbb{C}}K$ for all such $K$ is
called the $\mathbf{variation}$ of $f$ and denoted by Var$(f)$.
\begin{theorem}\label{tfae} Let $f: X\to Y$ be the Iitaka fibration as in
\cite[4.4]{fm00}. If the generic fiber $F$ of $f$ has a good minimal
model, then the following are equivalent:
\begin{enumerate}
\item $M_Y$ is numerically trivial.
\item $M_Y\sim_{\mathbb{Q}}0.$
\item $\kappa(Y,M_Y)=0.$
\item $\mbox{\rm{Var}}(f)=0.$
\end{enumerate}
\end{theorem}
\begin{proof}  (1)$\Longleftrightarrow$(2) is followed by
\cite[Theorem 3.5]{amb05}. The implication (2)$\Longrightarrow$(3)
is trivial. Since $F$ has a good minimal model, following
\cite[Theorem 1.1]{kaw85}, we have (3)$\Longleftrightarrow$(4) (cf.
\cite[Remark 3.9]{fuj03}). Finally, Fujino \cite[Theorem
3.11]{fuj03} proves the implication (4)$\Longrightarrow$(2).
\end{proof}

\section{Birational covering families of pure log canonical centers}
In this section, we construct a birational covering family of pure
log canonical centers.

 Recall that a subset $P$ of a variety $Y$ is called $\mathbf{countably}$ $\mathbf{dense}$ if it is
not contained in the union of countably many closed subsets of $Y$.
\begin{lemma}\label{family}
Let $(Y,\Delta)$ be a log pair, where $Y$ is projective and let $D$
be a big $\mathbb{Q}$-Cartier divisor on $Y$. Suppose that for every
point $y\in P$, where $P$ is a countably dense subset of \;$Y$, we
can find a pair $(\Delta_y,W_y)$ such that $W_y$ is a pure log
canonical center for $K_Y+\Delta+\Delta_y$ at $y$ and
$\Delta_y\sim_{\mathbb{Q}}D/w_y$ for some positive rational number
$w_y$. Then there exists a diagram
$$\xymatrix{Y' \ar[r]^{\pi}  \ar[d]^{\varphi} & Y\\
T}$$ such that $\varphi$ is a dominant morphism of normal projective
varieties with connected fibers and for a general fiber $V_t$ of
$\varphi$ there exists $y\in \varphi(V_t)$ so that $\varphi(V_t)$ is
a pure log canonical center for $K_Y+\Delta+\Delta_t$ with
$\Delta_t\sim_{\mathbb{Q}}D/w$ at $y$, for some weight $w$. Also
$\pi$ is a generically finite and dominant morphism of normal
varieties.
\end{lemma}
\begin{proof}
See \cite[Lemma 3.2]{mck02} or \cite[Lemma 3.2]{tod08}.
\end{proof}

\begin{lemma}[M$^{\rm{c}}$Kernan]\label{mck}
Let $(Y,\Delta)$ be a log pair, where $Y$ is a normal projective
variety of dimension $n'$. Let $D$ be a nef and big
 \;$\mathbb{Q}$-Cartier divisor. Let $(\Delta_t,V_t)$ be a covering
family of weight less than $w$ and dimension $k$.

If $(\Delta_t, V_t)$ is not birational then we may find a covering
family of $(\Gamma_s, W_s)$ of weight $w/(n'-k)$ and dimension $l$,
where either
\begin{enumerate}
\item $l>k$, or
\item $l<k$ and $(\Gamma_s,W_s)$ is a birational family.
\end{enumerate}
\end{lemma}

\begin{remark}
Lemma \ref{mck} still holds if we only assume that $D$ is big
instead of nef and big.
\end{remark}
\begin{proof}
See \cite[Lemma 4.2]{mck02}.
\end{proof}

\begin{corollary}\label{birational family} Let $(Y,\Delta)$ be a log pair, where $Y$ is a normal projective
variety of dimension $n'$. Let $D$ be a big $\mathbb{Q}$-Cartier
divisor. Let $(\Delta_t,V_t)$ be a covering family of weight $w$ and
dimension $k$. Then there exists a birational covering family of
$(\Gamma_s,W_s)$ of weight $w'\geq w/(n'-1)!$.

\end{corollary}
\begin{proof} This is immediate from Lemma \ref{mck}.
\end{proof}

By Lemma \ref{vz}, $K_Y+M_Y+B_Y$ is a big $\mathbb{Q}$-divisor on
$Y$, where $Y$ is a smooth projective variety of dimension $n'$, so
for each point $y\in Y$, we can find a pair $(D_y,V_y)$ such that
\begin{enumerate}
\item $D_y\sim_{\mathbb{Q}}\lambda(K_Y+M_Y+B_Y)$, for some rational
number $\lambda>0$,
\item $V_y$ is a pure log canonical center of $(Y,D_y)$ at $y$.
\end {enumerate}
Note that we can take the same $\lambda$ for every point in a
countably dense subset of $Y$ with dim$(V_y)=k$.
 Then by the previous corollary we obtain
a diagram
$$\xymatrix{Y' \ar[r]^{\pi} \ar[d]^{\varphi} & Y\\
T}
$$
such that
\begin{enumerate}
\item $\pi$ is birational and $\varphi$ is dominant.
\item Let $V_t=\pi(V'_t)$, where $V'_t$ is a general fiber of $\varphi$. Then there exists a $\mathbb{Q}$-divisor
$D_t\sim_{\mathbb{Q}}\lambda'(K_Y+M_Y+B_Y)$ on $Y$ such that $V_t$
is a pure log canonical center of $(Y,D_t)$ and $\lambda'\leq
\lambda (n'-1)!$.
\end{enumerate}

\begin{proposition}\label{fiber}
Let $f:X\rightarrow Y$ be the Iitaka fibration satisfying the
hypotheses of Theorem \ref{main thm}. Suppose that for
 any $y$ in a countably dense subset of \;$Y$, there is an effective
$\mathbb{Q}$-divisor $D_y\sim_{\mathbb{Q}}\lambda(K_Y+M_Y+B_Y)$ such
that $y\in \mbox{\rm{Non-klt}}(Y,D_y)$. Then there exists a diagram

$$\xymatrix{X' \ar[d]_{f'} \ar[r] & X \ar[d]^f\\
Y' \ar[d]_{\varphi} \ar[r]^{\pi} & Y\\
T }
$$

such that
\begin{enumerate}
\item $X'$ and $Y'$ are smooth projective varieties.
\item $\pi$ is birational, $\varphi$ is dominant with \mbox{\rm{dim}}$T\geq 0$ and $f'$ satisfies
the hypotheses of Theorem \ref{main thm}.
\item For any very general fiber $V'_t$ of $\varphi$, there exists
an effecitive $\mathbb{Q}$-divisor
$D'_t\sim_{\mathbb{Q}}\lambda'(K_{Y'}+M_{Y'}+B_{Y'})$ on $Y'$ such
that $V'_t$ is a pure log canonical center of $(Y',D'_t)$ and
$\lambda'\leq \lambda (n'-1)!$, where $n'=\mbox{\rm{dim}}Y$.
\end{enumerate}
\end{proposition}
\begin{proof}
By our discussions above, there exists a covering family
$Y'\overset{\varphi}{\rightarrow} T$ such that $Y'
\overset{\pi}{\rightarrow} Y$ is birational. Now replace $Y'$ by a
smooth model and let $X'$ be the resolution of the main component of
$X\times_{Y}Y'$. It is easy to see that $f'$ and $f$ have the same
generic fiber. Hence,
 (1) and (2) are satisfied. We only need to show (3).

Let $V_t=\pi(V'_t)$. By our assumptions and previous discussions,
there is an effective $\mathbb{Q}$-divisor
$D_t\sim_{\mathbb{Q}}\lambda'(K_Y+M_Y+B_Y)$ on $Y$ such that $V_t$
is a pure log canonical center of $(Y,D_t)$ and $\lambda' \leq
\lambda(n'-1)!$. Since $\pi$ is birational, for all $
m\in\mathbb{Z}_{>0}$ sufficiently divisible, we have
\begin{eqnarray*}
H^0(Y',\mathcal{O}_{Y'}(m(K_{Y'}+M_{Y'}+B_{Y'})))&\cong &
H^0(X',\mathcal{O}_{X'}(mK_{X'}))\\
&\cong& H^0(X,\mathcal{O}_X(mK_X))\\
&\cong& H^0(Y,\mathcal{O}_Y(m(K_Y+M_Y+B_Y))).
\end{eqnarray*}

So there is an effective $\mathbb{Q}$-divisor
$D'_t\sim_{\mathbb{Q}}\lambda'(K_{Y'}+M_{Y'}+B_{Y'})$ on $Y'$ such
that $\pi(D'_t)=D_t$. Since $V'_t$ is a very general fiber of
$\varphi$, $(Y',D'_t,V'_t)$ and $(Y,D_t,V_t)$ are isomorphic at the
generic point of $V'_t$. Therefore, $V'_t$ is a pure log canonical
center of $(Y',D'_t)$.
\end{proof}

\begin{lemma}[\mbox{\cite[Lemma 5.3]{mck02}}]\label{mckernan}
Let $(Y,\Delta)$ be a log pair and let $D$ be a $\mathbb{Q}$-divisor
of the form $A+E$ where $A$ is ample and $E$ is effective. Let
$(\Delta_t,V_t)$ be a covering family of weight greater than $w$ and
dimension $k$. Let $A_t$ be the restriction of $A$ to $V_t$. Suppose
that for all very general points $t\in U$ we may find a covering
family of $(\Gamma_{t,s},W_{t,s})$ on $V_t$ of weight, with respect
to $A_t$, greater than $w'$.

Then we may find a covering family of $(\Gamma_{s},W_s)$ of
dimension less than $k$ and weight
$$
w''=\frac{ww'}{w+w'}.$$ Further if both $(\Delta_t,V_t)$ and
$(\Gamma_{t,s},W_{t,s})$ are birational families then so is
$(\Gamma_s,W_s)$.
\end{lemma}

\section{Subadjunction}

In his fundamental paper \cite{kaw98}, Kawamata proves a remarkable
subadjunction theorem. An immediate consequence of this theorem is
that if $(X,D)$ is a log canonical pair, $V$ is a non-klt center of
$(X,D)$, then we have $(K_X+D)|_V \sim _{\mathbb{Q}} K_V+\Delta_V$,
where $\Delta_V$ is a pseudoeffective divisor on $V$. Actually, one
can prove a more precise result.

\begin{proposition}[Subadjunction]\label{sub}
Let $X$ be a normal variety and $D$ an effective
$\mathbb{Q}$-divisor on $X$ such that $(X,D)$ is a log pair. If
\:$V$ is a pure log canonical center of $(X,D)$ and $\nu :
V^{\nu}\rightarrow V $ is the normalization, then we have
$$(K_X+D)|_{V^{\nu}}\sim _\mathbb{Q}K_{V^{\nu}}+\Delta_{V^{\nu}},$$
where $\Delta_{V^{\nu}}$ is an effective $\mathbb{Q}$-divisor.
\end{proposition}

\begin{remark} Recently, Fujino and Gongyo \cite{fg10} prove the much
stronger result that if $(X,D)$ is an lc pair and $V$ is a minimal
non-klt center of $(X,D)$, then there exists an effective
$\mathbb{Q}$-divisor $\Delta_V$ on $V$ such that
$(K_X+D)|_V\sim_{\mathbb{Q}}K_V+\Delta_V$ and $(V,\Delta_V)$ is klt.
\end{remark}

This result depends on Ambro's results on the moduli
($\mathbf{b}$-)divisor associated to an lc-trivial fibration .


\begin{theorem}[Ambro]\label{ambro}
Let $f:(X,B)\rightarrow Y$ be an lc-trivial fibration such that the
generic geometric fiber $X_{\bar{\eta}}=X \times_Y
\mbox{\rm{Spec}}(\overline{k(Y)})$ is a projective variety and
$B_{\bar{\eta}}$ is effective. Then there exists a diagram\\

$$\begin{CD}
 (X,B) @.  @. (X^!,B^!)\\
 @VfVV   @.   @Vf^!VV\\
 Y @<\tau<<  \bar{Y} @>\varrho>> Y^!
\end{CD}$$\\\\
satisfying the following properties:

\begin{itemize}
\item $f^!:(X^!,B^!)\rightarrow Y^!$ is an lc-trivial fibration.
\item $\tau$ is generically finite and surjective and $\varrho$ is
surjective.
\item There exists a nonempty open subset $U\subset\bar{Y}$ and an
isomorphism

$$\xymatrix@1{
(X,B)\times_Y \bar{Y}|_U   \ar[rr]^{\simeq}  \ar[dr]& &  (X^!,B^!)
\times_{Y^!}\bar{Y}|_U \ar[dl] \\
 & U & }$$

\item Let $\mathbf{M}$ and $\mathbf{M}^!$ be the corresponding moduli
$\mathbb{Q}$-b-divisors. Then $\mathbf{M}^!$ is b-nef and big and
$\tau^*\mathbf{M}=\varrho^*(\mathbf{M}^!)$, which implies
$\mathbf{M}$ is b-nef and good. In particular, $\mathbf{M}$ is
$\mathbb{Q}$-linearly equivalent to an effective divisor.
\end{itemize}
\end{theorem}

\begin{proof}
See \cite[Theorem 3.3]{amb05}.
\end{proof}

Before giving the proof of \ref{sub}, we need the following useful
lemmas.

\begin{lemma}[Hacon]\label{hacon}
Let $X$ be a normal quasi-projective variety and $B$ a boundary
$\mathbb{R}$-divisor on $X$ such that $K_X + B$ is
$\mathbb{R}$-Cartier. Then, there exists a projective birational
morphism $f: Y\rightarrow X$ from a normal quasi-projective variety
$Y$ with the following properties.

\begin{enumerate}
\item $Y$ is $\mathbb{Q}$-factorial.

\item $a(E,X,B)\leq-1$ for every $f$-exceptional divisor $E$ on $Y$.

\item We put
$$B_Y=f^{-1}_*B+\sum_{E:\subset Ex(f)}E.$$
Then $(Y,B_Y)$ is dlt and
$$K_Y+B_Y = f^*(K_X+B)+\sum_{a(E,X,B)<-1}(a(E,X,B)+1)E.$$
In particular, if $(X,B)$ is lc, then $K_Y+B_Y = f^*(K_X+B).$
Moreover, if $(X,B)$ is dlt, then we can assume that $f$ is small,
that is, $f$ is an isomorphism in codimension one.
\end{enumerate}
\end{lemma}

\begin{proof}
See e.g. \cite[Theorem 10.4]{fuj09}.
\end{proof}

\begin{remark}\label{hacon'}
Lemma \ref{hacon} still holds if the coefficients of some components
of $B$ are greater than $1$. But we need to replace (3) by
\begin{enumerate}
\item[(3')]Let $$B_Y=f^{-1}_*B^{\leq 1}+\mbox{Supp}f^{-1}_*B^{>1}+\sum_{E:\subset
Ex(f)}E.$$ Then $(Y,B_Y)$ is dlt and
$$K_Y+B_Y=f^*(K_X+B)+\sum_{a(F,X,B)<-1}(a(F,X,B)+1)F.$$
\end{enumerate}

\end{remark}

\begin{lemma}[Adjunction for dlt pairs]\label{dlt}
Let $(X,D)$ be a dlt pair. We put $S=\lfloor D\rfloor$ and let
$S=\sum_{i\in I}S_i$ be the irreducible decompostion of $S$. Then,
$W$ is a non-klt centre for the pair $(X,D)$ with
$\mbox{\rm{codim}}_X W=k$ if and only if \ $W$ is an irreducible
component of \ $S_{i_1}\cap S_{i_2}\cap \cdots \cap S_{i_k} $ for
some $\{ i_1,i_2,\cdots ,i_k\}\subset I$. By adjunction, we obtain
$$K_{S_{i_1}}+\mbox{\rm{Diff}}(D-S_{i_1})=(K_X+D)|_{S_{i_1}},$$
and $(S_{i_1},\mbox{\rm{Diff}}(D-S_{i_1}))$ is dlt. Note that
$S_{i_1}$ is normal, $W$ is a non-klt center for the pair $(S_{i_1},
\mbox{\rm{Diff}}(D-S_{i_1}))$, $S_{i_j}|_{S_{i_1}}$ is a reduced
component of $\mbox{\rm{Diff}}(D-S_{i_1})$ for $2\leq j \leq k$, and
$W$ is an irreducible component of $(S_{i_2}|_{S_{i_1}})\cap
(S_{i_3}|_{S_{i_1}})\cap \cdots \cap (S_{i_k}|_{S_{i_1}}).$ By
applying adjunction $k$ times, we obtain a $\mathbb{Q}$-divisor
$\Delta \geq 0$ on $W$ such that
$$(K_X+D)|_{W}= K_W+\Delta$$
and $(W,\Delta)$ is dlt.
\end{lemma}

\begin{proof}
See \cite[Proposition 3.9.2]{cor07}.
\end{proof}

\begin{proof}[Proof of Proposition \ref{sub}]
Applying Lemma \ref{hacon} and Remark \ref{hacon'}, we may get a
morphism $f:Y \rightarrow X$ satisfying the properties of Lemma
\ref{hacon}. Let $D_Y=f^{-1}_*D^{\leq
1}+\mbox{Supp}f^{-1}_*D^{>1}+\sum_{E:\subset Ex(f)}E$. Then we have
$$f^*(K_X+D)=K_Y+D_Y-\sum_{a(F,X,D)<-1}(a(F,X,D)+1)F,$$  and the pair
$(Y,D_Y)$ is dlt. Since $V$ is a pure log canonical center of
$(X,D)$, $F$ is vertical over $V$ if $a(F,X,D)<-1$.

 Let $W$ be a minimal non-klt center of $(Y,D_Y)$ over the generic
point of $V$ and $\nu: V^{\nu} \rightarrow V$ the normalization of
$V$. We obtain the following diagram
$$\xymatrix{ &W  \ar[d]^g \ar[ld]_s \ar @{^{(}->}[rr]&  & Y \ar[d]^{f} \\
U \ar[r]^t& V^{\nu} \ar[r]^{\nu}& V \ar @{^{(}->}[r] & X }$$ where
$g: W\to V^{\nu}$ is the induced morphism and
$W\overset{s}{\rightarrow} U \overset{t}{\rightarrow} V^{\nu}$ is
the Stein factorization of $g$.

 By
Lemma \ref{dlt}, there exists a log pair $(W, \Delta _W)$, where
$\Delta_W\geq 0$, such that
$$K_W+\Delta_W\sim_{\mathbb{Q}}(K_Y+D_Y-\sum_{a(F,X,D)<-1}(a(F,X,D)+1)F)|_W \sim_{\mathbb{Q}}f^*(K_X+D)|_W,$$
and the non-klt centers of $(W, \Delta _W)$ are vertical over
$V^{\nu}$, so $(W,\Delta_W)$ has klt singularities over the generic
point of $V^{\nu}$. It follows that $(W,\Delta_{W})$ is klt over the
generic point of $U$. Moreover,
$$K_W+ \Delta_W \sim_{\mathbb{Q}} g^*((K_X+D)|_{V^{\nu}})\sim_{\mathbb{Q}}s^*((K_X+D)|_U).$$
Therefore, $s:(W,\Delta_W)\rightarrow U$ is an lc-trivial fibration
as defined in \cite[Definition 2.1]{amb04}.

We may write $(K_X+D)|_U\sim_{\mathbb{Q}}K_{U}+M+B$, where $M$ is
the moduli part and $B$ is the boundary part of this lc-trivial
fibration. Since $\Delta_W\geq 0$, $B\geq 0$. By Theorem
\ref{ambro}, we may assume that $M$ is effective. Let
$\Delta_{U}=M+B$, then,
$$(K_X+D)|_U \sim_{\mathbb{Q}} K_{U}+\Delta_{U}$$
and $\Delta_{U}\geq 0.$ Since $t:U\rightarrow V^{\nu}$ is finite and
$K_U+\Delta_{U}\sim_{\mathbb{Q}} t^*((K_X+D)|_{V^{\nu}})$, it is
easy to see that there exists an effective $\mathbb{Q}$-divisor
$\Delta_{V^{\nu}}$ on $V^{\nu}$ such that
$$(K_X+D)|_{V^{\nu}}\sim_{\mathbb{Q}}K_{V^{\nu}}+\Delta_{V^{\nu}}.
$$ \end{proof}

\section{Creating isolated non-klt centers}
\begin{proposition}\label{isolated}
Assume that Theorem \ref{main thm} holds for varieties of dimensions
$<n$. Let $f:X\rightarrow Y$ be the Iitaka fibration satisfying the
hypotheses of Theorem \ref{main thm} with $\mbox{\rm{dim}}X=n$ and \
$\mbox{\rm{dim}}Y=n'$. Then there exist positive constants $\alpha$
and $\beta$ depending on $n,b$ and $k$, such that for any very
general point $y\in Y$ there is an effective $\mathbb{Q}$-divisor
$D_y$ such that
\begin{enumerate}
\item $D_y\sim_{\mathbb{Q}}\lambda(K_Y+M_Y+B_Y)$, where $\lambda<
\dfrac{\alpha}{\mbox{\rm{vol}}(Y,K_Y+M_Y+B_Y)^{1/n'}}+\beta$;
\item $y$ is an isolated point of \;\mbox{\rm{Non-klt}}$(Y,D_y)$.
\end{enumerate}
\end{proposition}

\begin{proof}
 Take a very general point $y\in Y$. Since $K_Y+M_Y+B_Y$ is big, by the argument in the proof of
\cite[Theorem 6.2]{pac07b}, we can pick an effective
$\mathbb{Q}$-divisor $D_0\sim_{\mathbb{Q}}\lambda_0(K_Y+M_Y+B_Y)$
which has multiplicity $>n_0$ at $y$, where $n_0=n'$ and $\lambda_0<
n_0(\mbox{\rm{vol}}(Y,K_Y+M_Y+B_Y))^{-1/n_0}+\varepsilon_0$ with
$1\gg \varepsilon_0>0$. Hence there is a component $V_0$ of
Non-klt$(Y,D_0)$ passing through $y$. Multiplying $D_0$ by a
positive rational number $\leq 1$, we can assume that $V_0$ is a
pure log canonical center of $(Y,D_0)$.

By Proposition \ref{fiber}, we may replace $Y$ with a higher smooth
birational model such that there exists a morphism
$\varphi:Y\rightarrow T$ satisfying the properties of \ref{fiber}.
Therefore, the point $y$ is contained in a very general fiber $V_1$
of $\varphi$ and there is an effective $\mathbb{Q}$-divisor
$D_1\sim_{\mathbb{Q}}\lambda_1(K_Y+M_Y+B_Y)$ on $Y$ with
$\lambda_1\leq \lambda_0(n_0-1)!<
n_0!(\mbox{vol}(Y,K_Y+M_Y+B_Y))^{-1/n_0}+\varepsilon_0(n_0-1)!$,
such that $V_1$ is a pure log canonical center of $(Y,D_1)$.

 By Lemma \ref{canonical}, there is a log canonical
model $Y'$ of $(Y, M_Y+B_Y)$. Replacing $Y$ with a higher smooth
birational model, we may assume that there is a morphism $\phi:
Y\rightarrow Y'$. Let $M_{Y'}=\phi_*M_Y$ and $B_{Y'}=\phi_*B_Y$.
Then $K_{Y'}+M_{Y'}+B_{Y'}$ is $\mathbb{Q}$-Cartier and ample on
$Y'$.

By our assumption, the generic fiber of $f$ has a good minimal
model. Applying Theorem \ref{lai}, there exists a good minimal model
$X'$ of $X$. Replacing $X$ with a higher smooth birational model, we
may assume that there is a morphism $\psi: X\rightarrow X'$. Hence,
we obtain a diagram
$$\xymatrix{X \ar[d]_f \ar[r]^{\psi} & X' \ar@{-->}[d]^{f'}\\
Y\ar[d]_{\varphi} \ar[r]^{\phi} & Y'\\
T }
$$
where  $f'$ is the induced rational map.

\begin{remark}\label{rmk1}
The generic fiber of $f$ may have changed after running the Minimal
Model Program, so $f$ may not satisfy the hypotheses of Theorem
\ref{main thm}. But since our new $X$ is a higher birational model
of the original one, we do not change either $M_Y$ or $B_Y$ by the
Canonical Bundle Formula.
\end{remark}

\begin{lemma}\label{pullback}
We have the following:
\begin{enumerate}
\item $Y'$ is isomorphic to the weak canonical model $(X')^w$ of \;$X'$ in the sense that $$(X')^w=\mbox{\rm{Proj}}\bigoplus_{m\geq0}H^0(X',\mathcal{O}_{X'}(mK_{X'})).$$
\item $f'$ is a morphism and $K_{X'}\sim_{\mathbb{Q}}f'^*(K_{Y'}+M_{Y'}+B_{Y'})$.
\end{enumerate}
\end{lemma}
\begin{proof}


$X'$ is a good minimal model, so $X'$ admits a morphism to its weak
canonical model $(X')^w$. On the other hand, $K_{Y'}+M_{Y'}+B_{Y'}$
is ample on $Y'$, so
$$Y'=\mbox{\rm{Proj}}\bigoplus_{m\geq0}H^0(Y',\mathcal{O}_{Y'}(\lfloor
m(K_{Y'}+M_{Y'}+B_{Y'})\rfloor)).$$
 If $m\in
\mathbb{Z}_{>0}$ is sufficiently divisible, by the Canonical Bundle
Formula we have
\begin{eqnarray*}
H^0(X',\mathcal{O}_{X'}(mK_{X'}))&\cong &H^0(X,\mathcal{O}_X(mK_X))\\
 &\cong &H^0(Y,\mathcal{O}_Y(\lfloor
m(K_Y+M_Y+B_Y)\rfloor))\\
&\cong &H^0(Y',\mathcal{O}_{Y'}(\lfloor
m(K_{Y'}+M_{Y'}+B_{Y'})\rfloor)).
\end{eqnarray*}
Hence $Y'$ is the weak canonical model of $X'$ and (2) follows from
(1).
\end{proof}
Now let $y'=\phi(y)$, $V'_1=\phi(V_1)$, and $D'_1=\phi_*(D_1)$ and
let $n_1=\mbox{dim}V_1=\mbox{dim}V'_1$. Since $V_1$ is a pure log
canonical center of $(Y,D_1)$ and $y'$ is very general, it follows
that $V'_1$ is a pure log canonical center of
$(Y',M_{Y'}+B_{Y'}+D'_1)$ at $y'$. Let $W_1=f^{-1}(V_1)$,
$W'_1=f'^{-1}(V'_1)$, $V^{\nu}_1$ the normalization of $V'_1$,
$W^{\nu}_1$ the normalization of $W'_1$ and
$\gamma:W^{\nu}_1\rightarrow V^{\nu}_1$ the induced morphism. We
have the following diagram
$$ \xymatrix@!0{ W_1 \ar [rr] \ar [dd]_{f_{W_1}} \ar@{^{(}->}[dr] &
&W_1^{\nu} \ar [rr] \ar '[d]^{\gamma}[dd]&    &
W_1'\ar@{^{(}->}[dr]  \ar '[d][dd] \\
& X \ar [rrrr] \ar [dd]&    &   &   &X' \ar [dd] \\
V_1 \ar '[r][rr] \ar@{^{(}->}[dr]\ar[dd] &  & V_1^{\nu} \ar [rr] &
&V_1'
\ar@{^{(}->}[dr]\\
&   Y \ar [rrrr] \ar [dd] & & &  &Y'\\
t_1 \ar @{} [dr]|\in\\
& T }
$$
By Lemma \ref{iitaka} and Lemma \ref{restriction}, the morphism
$f_{W_1}: W_1\rightarrow V_1$ is the Iitaka fibration of
$(W_1,K_{W_1})$ and  the moduli part $M_{V_1}$  of $f_{W_1}$ is
$\mathbb{Q}$-linearly equivalent to the restriction of $M_Y$ to
$V_1$. Thus we can assume that $f_{W_1}$ satisfies the hypotheses of
Theorem \ref{main thm}.

\begin{remark}
As in Remark \ref{rmk1}, the generic fiber of $f_{W_1}$ may be
different from the original one. However this does not affect the
computation of $M_{V_1}$ and $B_{V_1}$.

\end{remark}

\begin{lemma} There exists a constant $\delta>0$
depending on $n-1,b$ and $k$, such that
$\mbox{\rm{vol}}(V_1,K_{V_1}+M_{V_1}+B_{V_1})\geq \delta.$
\end{lemma}
\begin{proof}
Since $\mbox{dim}W_1<n$, by our assumptions in Proposition
\ref{isolated}, there exists a positive integer $m_1$ depending on
$n-1,b$ and $k$, such that $\phi_{m_1(K_{V_1}+M_{V_1}+B_{V_1})}$
gives a birational map. Then
$\mbox{\rm{vol}}(V_1,m_1(K_{V_1}+M_{V_1}+B_{V_1}))\geq 1$ by Lemma
\ref{vol}. Therefore,
\begin{eqnarray*}
\mbox{\rm{vol}}(V_1,K_{V_1}+M_{V_1}+B_{V_1})&=&
\frac{1}{m_1^{n_1}}\mbox{\rm{vol}}(V_1,m_1(K_{V_1}+M_{V_1}+B_{V_1}))\\
&\geq& \frac{1}{m_1^{n_1}}\\
&\geq& \frac{1}{m_1^{n-1}}.
\end{eqnarray*}
Now let $\delta$ be $1/m_1^{n-1}$.
\end{proof}

We have the following fact.
\begin{lemma}\label{vol-lamma}
$\mbox{\rm{vol}}(V'_1,(K_{Y'}+M_{Y'}+B_{Y'}+D'_1)|_{V'_1})\geq
\delta.$
\end{lemma}
\begin{proof}
By Lemma \ref{pullback}, we have
$K_{X'}\sim_{\mathbb{Q}}f'^*(K_{Y'}+M_{Y'}+B_{Y'})$. $V_1'$ is a
pure log canonical center of $(Y',M_{Y'}+B_{Y'}+D'_1)$ and $y'$ is a
very general point of $Y'$, so $W'_1$ is a pure log canonical center
of $(X',f'^*D'_1)$.

 By Proposition \ref{sub}, there exists an effective
 $\mathbb{Q}$-divisor $\Delta_{W^{\nu}_1}$ on $W^{\nu}_1$, such that
$$(K_{X'}+f'^*D'_1)|_{W^{\nu}_1}\sim_{\mathbb{Q}}K_{W^{\nu}_1}+\Delta_{W^{\nu}_1}.$$
On the other hand,
$$(K_{X'}+f'^*D'_1)|_{W^{\nu}_1}\sim_{\mathbb{Q}}\gamma^*((K_{Y'}+M_{Y'}+B_{Y'}+D_1')|_{V_1^{\nu}}).$$
For all $  m \in \mathbb{Z}_{>0}$ sufficiently divisible, by the
Projection Formula we have
$$h^0(W^{\nu}_1,
\mathcal{O}_{W^{\nu}_1}(m(K_{W^{\nu}_1}+\Delta_{W_1^{\nu}})))=h^0(V^{\nu}_1,\mathcal{O}_{V^{\nu}_1}(m(K_{Y'}+M_{Y'}+B_{Y'}+D'_1)|_{V^{\nu}_1})).\eqno{(*)}$$
By the Canonical Bundle Formula,
$$
h^0(W_1,\mathcal{O}_{W_1}(mK_{W_1}))
=h^0(V_1,\mathcal{O}_{V_1}(m(K_{V_1}+M_{V_1}+B_{V_1}))). \eqno{(**)}
$$
Since $W_1$ is smooth and $\Delta_{W^{\nu}_1}\geq 0$, it follows
that
$$h^0(W^{\nu}_1,\mathcal{O}_{W^{\nu}_1}(m(K_{W^{\nu}_1}+\Delta_{W_1^{\nu}})))\geq
h^0(W_1,\mathcal{O}_{W_1}(mK_{W_1})).$$ Therefore, by equations
$(*)$ and $(**)$,
$$h^0(V^{\nu}_1,\mathcal{O}_{V^{\nu}_1}(m(K_{Y'}+M_{Y'}+B_{Y'}+D'_1)|_{V_1^{\nu}}))\geq h^0(V_1,\mathcal{O}_{V_1}(m(K_{V_1}+M_{V_1}+B_{V_1}))),$$
which implies
$$\mbox{\rm{vol}}(V^{\nu}_1,(K_{Y'}+M_{Y'}+B_{Y'}+D'_1)|_{V^{\nu}_1})\geq \mbox{\rm{vol}}(V_1,K_{V_1}+M_{V_1}+B_{V_1}).$$
Note that the normalization $\nu:V^{\nu}_1\rightarrow V'_1$ is
birational. Thus we have
\begin{eqnarray*}
\mbox{\rm{vol}}(V'_1,(K_{Y'}+M_{Y'}+B_{Y'}+D'_1)|_{V'_1}) &=&
\mbox{\rm{vol}}(V^{\nu}_1,(K_{Y'}+M_{Y'}+B_{Y'}+D'_1)|_{V^{\nu}_1})\\
& \geq & \mbox{\rm{vol}}(V_1,K_{V_1}+M_{V_1}+B_{V_1})\\
&\geq & \delta.
\end{eqnarray*}
\end{proof}

Let $\phi_{V_1}:V_1\to V_1'$ be the restriction of $\phi$ to $V_1$.
We have
$$\phi^*(K_{Y'}+M_{Y'}+B_{Y'})|_{V_1}\sim_{\mathbb{Q}}\phi^*_{V_1}((K_{Y'}+M_{Y'}+B_{Y'})|_{V_1'}).$$
Recall that $D'_1\sim_{\mathbb{Q}}\lambda_1(K_{Y'}+M_{Y'}+B_{Y'})$,
so by Lemma \ref{vol-lamma} it follows that
\begin{eqnarray*}
\mbox{vol}(V_1,\phi^*(K_{Y'}+M_{Y'}+B_{Y'})|_{V_1})&=&\mbox{vol}(V_1',(K_{Y'}+M_{Y'}+B_{Y'})|_{V'_1})\\
&=&\frac{\mbox{vol}(V_1',(K_{Y'}+M_{Y'}+B_{Y'}+D_1')|_{V'_1})}{(1+\lambda_1)^{n_1}}\\
&\geq& \frac{\delta}{(1+\lambda_1)^{n_1}}.
\end{eqnarray*}
Hence for any very general fiber $V_t$ of $\varphi$, we always have
$$\mbox{vol}(V_t,\phi^*(K_{Y'}+M_{Y'}+B_{Y'})|_{V_t})\geq
\delta
 (1+\lambda_1)^{-n_1}.$$ Then for any point $p\in V_t$, there
exists an effective $\mathbb{Q}$-divisor
$E_{t,p}\sim_{\mathbb{Q}}\lambda_{t,p}(\phi^*(K_{Y'}+M_{Y'}+B_{Y'})|_{V_t})$
on $V_t$ such that $\mbox{mult}_pE_{t,p}>n_1$ and
\begin{eqnarray*}
\lambda_{t,p}&<&
\frac{n_1}{\mbox{vol}(V_t,\phi^*(K_{Y'}+M_{Y'}+B_{Y'})|_{V_t})^{1/n_1}}+\varepsilon_1\\
&<&\frac{n_0!n_1}{\delta^{1/n_1}
\mbox{vol}(Y,K_Y+M_Y+B_Y)^{1/n_0}}+(1+\varepsilon_0(n_0-1)!)\frac{n_1}{\delta^{1/n_1}}+\varepsilon_1
\end{eqnarray*}
where $0<\varepsilon_1\ll 1$. This implies that there is a component
of $\mbox{Non-klt}(V_t, E_{t,p})$ passing through $p$\,. Multiplying
$E_{t,p}$ by a positive rational number $\leq 1$, we can assume that
$p$ is contained in a pure log canonical center of $(V_t,E_{t,p})$.

Applying Lemma \ref{family} and Corollary \ref{birational family},
there exists a birational covering family of
$(\Gamma_{t,s},W_{t,s})$ on $V_t$ of weight $w'$ with respect to
$\phi^*(K_{Y'}+M_{Y'}+B_{Y'})|_{V_t}$ such that
$\Gamma_{t,s}\sim_{\mathbb{Q}}(1/w')\phi^*(K_{Y'}+M_{Y'}+B_{Y'})|_{V_t}$
and the image of $W_{t,s}$ on $V_t$ is a pure log canonical center
of $(V_t,\Gamma_{t,s})$, where
$$\frac{1}{w'}<\frac{n_0!n_1!}{\delta^{1/n_1}
\mbox{vol}(Y,K_Y+M_Y+B_Y)^{1/n_0}}+(1+\varepsilon_0(n_0-1)!)\frac{n_1!}{\delta^{1/n_1}}+\varepsilon_1(n_1-1)!.$$

By Lemma \ref{mckernan}, we can find a new birational covering
family of $(D'_s,V''_s)$ on $Y'$ of dimension less than $n_1$ and
weight $w''$ such that
\begin{eqnarray*}
\frac{1}{w''}&=&\lambda_1+\frac{1}{w'}\\
&<&\frac{n_0!n_1!\delta^{-1/n_1}+n_0!}{\mbox{vol}(Y,K_Y+M_Y+B_Y)^{1/n_0}}\\
&&+(1+\varepsilon_0(n_0-1)!)\frac{n_1!}{\delta^{1/n_1}}+\varepsilon_1(n_1-1)!+\varepsilon_0(n_0-1)!.
\end{eqnarray*}
Therefore, we obtain the following diagram
$$
\xymatrix{Y'' \ar[r]^{\phi''} \ar[d]_{\varphi''} & Y'\\
S }$$ where $\phi''$ is birational and $\varphi''$ is surjective.
For the very general point $y'\in Y'$, there are an effective
$\mathbb{Q}$-divisor
$D'_s\sim_{\mathbb{Q}}\lambda_2(K_{Y'}+M_{Y'}+B_{Y'})$ on $Y'$ with
$\lambda_2=1/w''$ and a very general fiber $V''_s$ of $\varphi''$
such that $V'_2=\phi''(V''_s)$ is a pure log canonical center of
$(Y',M_{Y'}+B_{Y'}+D'_s)$ at $y'$ with dim$V'_2<\mbox{dim}V'_1=n_1$.
Replacing $Y''$ with the common higher smooth model of $Y,Y'$ and
$Y''$, we can assume that $Y''$ is smooth and the dimension of any
very general fiber of $\varphi'': Y''\to S$ is strictly less than
that of $\varphi: Y\to T$. The moduli part $M_{Y''}$ on $Y''$ is
still $\mathbb{Q}$-linearly trivial, since it is the pullback of
$M_Y$.

Repeating above procedure at most $n'-1$ times, there exists an
effective $\mathbb{Q}$-divisor
$D'\sim_{\mathbb{Q}}\lambda(K_{Y'}+M_{Y'}+B_{Y'})$ on $Y'$ with
$\lambda< \alpha(\mbox{\rm{vol}}(Y,K_Y+M_Y+B_Y))^{-1/n'}+\beta$,
where $\alpha$ and $\beta$ depend only on $n,k$ and $b$, such that
$y'$ is a pure log canonical center of $(Y',M_{Y'}+B_{Y'}+D')$. By
the standard tie-breaking technique, we can assume that $y'$ is the
unique non-klt center of $(Y',M_{Y'}+B_{Y'}+D')$ on a neighborhood
of $y'$, i.e. $y'$ is an isolated point of
Non-klt$(Y',M_{Y'}+B_{Y'}+D')$. Since $Y'$ and $Y$ are birational,
there is a unique effective $\mathbb{Q}$-divisor
$D_y\sim_{\mathbb{Q}}\lambda(K_Y+M_Y+B_Y)$ on $Y$ such that
$\phi_*(D_y)=D'$. Then $D_y$ satisfies the requirements in
Proposition \ref{isolated}. This completes the proof.
\end{proof}

\begin{remark}\label{no1} If we assume Theorem \ref{main thm} without the
hypothesis (1) holds for varieties of dimension $<n$ (i.e. we do not
assume that $M_Y \sim_{\mathbb{Q}}0$), then for any Iitaka fibration
$f: X\to Y$ satisfying the hypotheses (2) and (3) of Theorem
\ref{main thm} with $\mbox{dim}X=n$ and $\mbox{dim}Y=n'$, the
conclusion of Proposition \ref{isolated} still holds. Therefore, if
Theorem \ref{n-1 thm} holds for varieties of dimension $<n$, then
for any $n$-dimensional variety $X$ of Kodaira dimension $n-1$ with
Iitaka fibration $f:X\to Y$, there exist positive constants $\alpha$
and $\beta$ depending only on $n$ such that for any very general
point $y\in Y$ there is an effective $\mathbb{Q}$-divisor $D_y$
satisfying (1) and (2) of Proposition \ref{isolated}.
\end{remark}

\section{Proof of \ref{main thm} and \ref{n-1 thm}}
\begin{lemma}\label{global sections}
Let $f:X\rightarrow Y$ be the Iitaka fibration satisfying the
hypotheses of Theorem \ref{main thm}. Let $m_0$ be a positive
integer and assume that for any very general point $y\in Y$, there
exists an effective $\mathbb{Q}$-divisor
$D_y\sim_{\mathbb{Q}}\lambda(K_Y+M_Y+B_Y)$ where $\lambda\leq
m_0-1$, such that $y$ is an isolated point in
\mbox{\rm{Non-klt}}$(Y,D_y)$. Then for all $m\geq m_0$ such that
$mM_Y$ is an integral divisor, i.e. $m$ is divisible by $bN$, we
have $h^0(X,\mathcal{O}_X(mK_X))>0$ and moreover, if $m\geq 2m_0$,
then $h^0(X,\mathcal{O}_X(mK_X))\geq 2$.
\end{lemma}
\begin{proof}
Since $K_Y+M_Y+B_Y$ is big, there exist an ample
$\mathbb{Q}$-divisor $H$ and an effective $\mathbb{Q}$-divisor $G$
on $Y$ such that $K_Y+M_Y+B_Y\sim_{\mathbb{Q}}H+G$. Pick a very
general point $y\in Y$ not contained in the support of $G+B_Y$. By
Lemma \ref{vz}, the divisor $(\lfloor mB_Y\rfloor-(m-1)B_Y)$ is
effective. Let $D'_y=D_y+(m-1-\lambda)G+\lfloor
mB_Y\rfloor-(m-1)B_Y$. Then
$$
\lfloor m(K_Y+M_Y+B_Y) \rfloor - K_Y-D'_y\sim_{\mathbb{Q}}
(m-1-\lambda)H+M_Y$$ is ample so that $H^1(Y,\mathcal{O}_Y(\lfloor
m(K_Y+M_Y+B_Y) \rfloor)\otimes \mathcal{J}(Y,D'_y))=0$.

Consider the short exact sequence of coherent sheaves on $Y$
$$0\rightarrow \mathcal{O}_Y(\lfloor
m(K_Y+M_Y+B_Y) \rfloor)\otimes \mathcal{J}(Y,D'_y) \rightarrow
\mathcal{O}_Y(\lfloor m(K_Y+M_Y+B_Y) \rfloor) \rightarrow
\mathcal{Q} \rightarrow 0$$ where $\mathcal{Q}$ denotes the
corresponding quotient. By the discussion above, the map
$$H^0(Y,\mathcal{O}_Y(\lfloor m(K_Y+M_Y+B_Y) \rfloor))\rightarrow H^0(Y,\mathcal{Q})$$
is surjective. Since $y$ is an isolated point in Non-klt$(Y,D'_y)$,
$\mathbb{C}_y$ is a direct summand of $H^0(Y,\mathcal{Q})$. Thus, we
have
$$
h^0(X,\mathcal{O}_X(mK_X))=h^0(Y,\mathcal{O}_Y(\lfloor
m(K_Y+M_Y+B_Y) \rfloor))>0.$$

Pick a very general point $y_1\in Y$. Then there is an effective
$\mathbb{Q}$-divisor $D_{y_1}\sim_{\mathbb{Q}}\lambda(K_Y+M_Y+B_Y)$
such that $y_1$ is an isolated point in Non-klt$(Y,D_{y_1})$. Now we
may pick a very general point $y_2\in Y$ not contained in the
support of $D_{y_1}$, and pick a very general divisor
$D_{y_2}\sim_{\mathbb{Q}}\lambda(K_Y+M_Y+B_Y)$ such that $y_2$ is an
isolated point in Non-klt$(Y,D_{y_2})$ and $y_1$ is not contained in
the support of $D_{y_2}$. Hence $y_1$ and $y_2$ are isolated points
in Non-klt$(Y,D_{y_1}+D_{y_2})$. Then
$h^0(X,\mathcal{O}_X(mK_X))\geq 2$ by an argument similar to the
discussion above.
\end{proof}

\begin{lemma}\label{birational map}
Let $f:X\rightarrow Y$ be the Iitaka fibration satisfying the
hypotheses of Theorem \ref{main thm}. Let $m'_0$ be a positive
integer divisible by $bN$. Assume that $h^0(X,mK_X)\geq 2$ for all
$m \geq m'_0$ such that $m$ is divisible by $bN$. Let $X'\rightarrow
Y'\rightarrow \mathbb{P}^1$ be any morphism induced by sections of
$\mathcal{O}_X(m_0'K_X)$ on an appropriate birational model
$f':X'\rightarrow Y'$ of $f:X\rightarrow Y$. Let $p\in \mathbb{P}^1$
be a very general point. $f_W:W\rightarrow V$ denotes the
restriction of $f'$ to the fiber over $p$. If there is a positive
integer $s$ divisible by $bN$ such that $|sK_W|$ induces the Iitaka
fibration for any very general point $p$, then $|tK_X|$ induces the
Iitaka fibration for all $t\geq m'_0(2s+2)+s$ such that $t$ is
divisible by $bN$.
$$
\xymatrix{ W \ar[d]^{f_W} \ar @{^{(}->}[r] &X' \ar[d]^{f'} \ar[r] & X \ar[d]^f\\
V \ar[d] \ar @{^{(}->} [r] &Y' \ar[d] \ar [r] &  Y\\
p \ar @{} [r]|{\in} & \mathbb{P}^1 }$$
\end{lemma}
\begin{proof}
Following [Kol86, Theorem 4.6] and its proof, $|(m'_0(2s+1)+s)K_X|$
gives the Iitaka fibration. Since $mK_X$ is effective for all $m\geq
m_0'$ such that $m$ is divisible by $bN$, the assertion follows.
\end{proof}

\begin{proof}[Proof of Theorem \ref{main thm}]

 Since the moduli
part is $\mathbb{Q}$-linearly trivial by Theorem \ref{tfae}, we
always have
$\mbox{\rm{vol}}(Y,K_Y+M_Y+B_Y)=\mbox{\rm{vol}}(Y,K_Y+B_Y)$. The
proof is by induction on the dimension of $X$. It is well known that
the theorem holds for $n = 1$. Assume that the
theorem holds when $\mbox{\rm{dim}}X\leq n-1$. 
Let $f : X \rightarrow Y$ be the Iitaka firation satisfying the
hypotheses of Theorem \ref{main thm} with dim$X = n$ and dim$Y=n'$.
By Proposition \ref{isolated}, for any very general point $y\in Y$ ,
there exists an effective $\mathbb{Q}$-divisor $D_y
\sim_{\mathbb{Q}} \lambda (K_Y + M_Y+B_Y )$ with $\lambda<
\alpha(\mbox{vol}(Y,K_Y +M_Y+B_Y ))^{-1/n'} + \beta$, where $\alpha$
and $\beta$ are two positive constants depending only on $n,b$ and
$k$, such that $y$ is an isolated point in Non-klt$(Y,D_ y)$.

 If $\mbox{\rm{vol}}(Y,K_Y+M_Y+B_Y)=\mbox{\rm{vol}}(Y,K_Y  +B_Y) \geq 1$
, Proposition \ref{isolated}, Lemma \ref{global sections} and Lemma
\ref{birational map} imply that there exists an positive integer
$m_n$ only depending on $n,b$ and $k$ such that $mK_X$ gives the
Iitaka fibration if $m \geq m_n$ and divisible by $bN$.

Now we prove the case when
$\mbox{\rm{vol}}(Y,K_Y+M_Y+B_Y)=\mbox{\rm{vol}}(Y,K_Y  +B_Y)< 1$. By
induction, there exists a positive integer $s$ such that $|sK_W|$
gives the Iitaka fibration for all $W$ with dim$W\leq n-1$
satisfying the hypotheses of Theorem \ref{main thm}. By Proposition
\ref{isolated}, Lemma \ref{global sections} and Lemma
\ref{birational map}, \,$|mK_X|$ induces the Iitaka fibration, for
$$m=8bNs\lceil \frac{\alpha}{\mbox{\rm{vol}}(Y,K_Y+M_Y+B_Y)^{1/n'}}+\beta+1 \rceil,$$
so $\phi_{m(K_Y+M_Y+B_Y)}$ gives a birational map. As $mM_Y$ is a
$\mathbb{Q}$-linearly trivial Cartier divisor,
$\phi_{K_Y+(2n'+1)m(K_Y+B_Y)}$ is also birational by Lemma
\ref{d-bir}. We have
\begin{eqnarray*}
\mbox{\rm{vol}}(Y,(2n'+1)m(K_Y+B_Y)) &=&
(2n'+1)^{n'}m^{n'}\mbox{\rm{vol}}(Y,K_Y+B_Y)\\
&\leq& (2n'+1)^{n'}(8bNs)^{n'}(\alpha+\beta+2)^{n'}\\
&\leq& (2n+1)^{n}(8bNs)^{n}(\alpha+\beta+2)^{n}.
\end{eqnarray*}
It follows that there is a constant $A$ such that
vol$(Y,(2n'+1)m(K_Y+B_Y))\leq A$. Then Lemma \ref{vz} and Theorem
\ref{log-bound} imply that the set of such log pairs $(Y,B_Y)$ is
log birationally bounded.

By Theorem \ref{dcc}, there exists a constant $\delta_{n}>0$ such
that
$$\mbox{\rm{vol}}(Y,K_Y+B_Y)\geq \delta_n.$$
So we are done by applying Proposition \ref{isolated}, Lemma
\ref{global sections} and Lemma \ref{birational map} again.
\end{proof}

\begin{proof}[Proof of Theorem \ref{n-1 thm}] By Remark
\ref{no1}, Lemma \ref{global sections}, Lemma \ref{birational map}
and the argument in the proof of Theorem \ref{main thm}, we only
need to show that vol$(Y,K_Y+M_Y+B_Y)$ is bounded from below.

Since Kodaira dimension of $X$ is $n-1$, the general fiber of $f$ is
an elliptic curve. The $j$-invariant defines a rational map
$J:Y\dashrightarrow \mathbb{P}^1$. Replacing $Y$ by a higher model,
we may assume that $J$ is a morphism. Then by \cite[7.16]{ps07}, we
have $$
M_Y\sim_{\mathbb{Q}}\frac{1}{12}J^*(\mathcal{O}_{\mathbb{P}^1}(1)).$$
$J^*(\mathcal{O}_{\mathbb{P}^1}(1))$ is base point free on $Y$.
Picking a general member $D_Y\in
|J^*(\mathcal{O}_{\mathbb{P}^1}(1))|$, we can assume that
$\frac{1}{12}D_Y+B_Y$ is simple normal crossings and
$$K_Y+M_Y+B_Y\sim_{\mathbb{Q}}K_Y+\frac{1}{12}D_Y+B_Y.$$

If
$\mbox{vol}(Y,K_Y+\frac{1}{12}D_Y+B_Y)=\mbox{vol}(Y,K_Y+M_Y+B_Y)\geq
1$, we are done. So we can assume that
$\mbox{vol}(Y,K_Y+\frac{1}{12}D_Y+B_Y)=\mbox{vol}(Y,K_Y+M_Y+B_Y)<
1$. By the same argument in the proof of Theorem \ref{main thm}, it
is easy to see that the set
$\{\mbox{vol}(Y,K_Y+\frac{1}{12}D_Y+B_Y)\}$ satisfies the $DCC$.
Therefore, $\mbox{vol}(Y,K_Y+M_Y+B_Y)$ is bounded from below.
\end{proof}
\begin{remark} In \cite[Conjecture 7.13]{ps07}, Prokhorov and
Shokurov list a series of conjectures concerning the effective
Iitaka fibration problem. If one can prove the effective adjunction
conjecture (\cite[Conjecture 7.13(3)]{ps07}), i.e. there exists a
positive integer $I$ depending only on the dimension of $X$ such
that $IM_Y$ is linearly equivalent to a base point free divisor on
$Y$, then by Remark \ref{no1} and the argument in the proof of 1.3,
one can prove Theorem \ref{main thm} without the hypothesis (1).
\end{remark}

\bibliographystyle{amsalpha}

\end{document}